\documentclass[english]{amsart}

\usepackage{esint}
\usepackage[svgnames]{xcolor} 
\usepackage[colorlinks,citecolor=red,pagebackref,hypertexnames=false,breaklinks]{hyperref}
\usepackage{pgf,tikz}
\usepackage{pdfsync}

\usepackage[bottom]{footmisc}
\usepackage{dsfont}
\usepackage{url}
\usepackage[utf8]{inputenc}
\usepackage[T1]{fontenc}
\usepackage{lmodern}
\usepackage{babel}
\usepackage{mathtools}  
\usepackage{amssymb}
\usepackage{lipsum}
\usepackage{mathrsfs}
\usepackage{color}
\usepackage{bbm}

\usepackage{epsf,graphicx,epsfig,cite,latexsym}
\usepackage[export]{adjustbox}
\usepackage{mathtools}
\usepackage{latexsym, amsmath, amssymb, a4, epsfig}
\usepackage[T1]{fontenc}

\usepackage{listings}

\usepackage{algorithm}
\usepackage{algpseudocode}

 \usepackage{mathrsfs}
\newcommand{\no}[1]{#1}
\renewcommand{\no}[1]{}
\usepackage{color}

\usepackage{tikz}
\usetikzlibrary{calc,trees,positioning,arrows,chains,shapes.geometric,decorations.pathreplacing,
decorations.pathmorphing,shapes,matrix,shapes.symbols}

\tikzset{
>=stealth',
  punktchain/.style={
    rectangle, 
    rounded corners, 
    draw=black, very thick,
    text width=10em, 
    minimum height=3em, 
    text centered, 
    on chain},
  line/.style={draw, thick, <-},
  element/.style={
    tape,
    top color=white,
    bottom color=blue!50!black!60!,
    minimum width=8em,
    draw=blue!40!black!90, very thick,
    text width=10em, 
    minimum height=3.5em, 
    text centered, 
    on chain},
  every join/.style={->, thick,shorten >=1pt},
  decoration={brace},
  tuborg/.style={decorate},
  tubnode/.style={midway, right=2pt},
}

\newtheorem{lemma}{Lemma}[section]
\newtheorem{remark}{Remark}[section]
\newtheorem{definition}{Definition}[section]

\newtheorem{proposition}{Proposition}[section]
\newtheorem{theorem}{Theorem}[section]

\def\F{\mathcal{F}}

\def\a{{\mathsf a}}
\def\ta{{\tilde{\mathsf a}}}
\def\A{{\mathsf A}}
\def\L{{\mathsf L_{\mathsf a}}}
\def\tL{{\mathsf L_{\tilde{\mathsf a}}}}
\def\F{{\mathsf F}}
\def\la{\langle}
\def\ra{\rangle}

\newcommand{\be}{\begin{equation}}
\newcommand{\ee}{\end{equation}}
\newcommand{\ba}{\begin{array}}
\newcommand{\ea}{\end{array}}
\newcommand{\bea}{\begin{eqnarray*}}
\newcommand{\eea}{\end{eqnarray*}}
\newcommand{\bean}{\begin{eqnarray}}
\newcommand{\eean}{\end{eqnarray}}

\def\div{\operatorname{div}}

\date{\today}

\numberwithin{algorithm}{section}
\numberwithin{figure}{section}

\title[]{Coefficient identification in parabolic equations  with final data}

\author{Faouzi Triki}

\address{Faouzi Triki,  Laboratoire Jean Kuntzmann,  UMR CNRS 5224, 
Universit\'e  Grenoble-Alpes, 700 Avenue Centrale,
38401 Saint-Martin-d'H\`eres, France}

\email{faouzi.triki@univ-grenoble-alpes.fr}

\thanks{
This work  is supported in part by the
 grant ANR-17-CE40-0029 of the French National Research Agency ANR (project MultiOnde).}

\subjclass{Primary: 35R30, 35K20.}
\keywords{inverse  problem,  uniqueness, stability estimate,  parabolic equation, final data}
\begin{document}

\lstset{language=Matlab,%
    breaklines=true,%
    morekeywords={matlab2tikz},
    keywordstyle=\color{blue},%
    morekeywords=[2]{1}, keywordstyle=[2]{\color{black}},
    identifierstyle=\color{black},%
    stringstyle=\color{mylilas},
    commentstyle=\color{mygreen},%
    showstringspaces=false,
    numbers=left,%
    numberstyle={\tiny \color{black}},
    numbersep=9pt, 
    emph=[1]{for,end,break},emphstyle=[1]\color{red}, 
}

\begin{abstract}

In this work we determine the second-order coefficient
in a parabolic equation from the knowledge of a single final data. Under assumptions on the concentration of eigenvalues 
of the associated elliptic operator, and   the initial state,  we show  the uniqueness  of solution,  and we derive a Lipschitz stability 
estimate for the inversion  when  the final time is large enough.  The Lipschitz stability constant grows exponentially  with respect to the final time, which makes the inversion  ill-posed.  The proof of the stability estimate is based on a spectral decomposition
of the solution to the parabolic equation in terms of the eigenfunctions of the  associated elliptic operator, and an ad hoc method
 to solve a nonlinear stationary transport equation that is itself of interest.
 \end{abstract}

\maketitle

\section{Introduction and main results}

Let $\Omega$ be a $C^3$ bounded domain of $\mathbb{R}^n$, $n=2,3$, with a  boundary $\Gamma$. Let $\nu(x)$ be the outward unitary normal vector at $x\in \Gamma$.  For $ \mathsf a_+ >1$,  a fixed constant, and $\mathsf a_0 \in C^1(\Gamma)$,   a given function,
set

\[
\mathsf A = \left\{ a  \in C^1(\overline \Omega): \; \;   1\leq \mathsf a(x); \;\; \mathsf a|_{\Gamma} = \mathsf a_0 ;\;\;
\| \a\|_{C^1(\overline \Omega)}  \leq \mathsf a_+  \right\}.
\]
Consider, for $u_0\in L^2(\Omega)$ and $\mathsf a \in \mathsf A$,  the following  initial-boundary value problem 

\bean \label{mainequation} \left\{ \ba{lllcc}
u_t-\div(\mathsf a\nabla u) = 0 &\textrm{in} & \Omega \times ]0, +\infty[,\\
u=0  &\textrm{on}&   \Omega \times ]0,+\infty[,\\
u= u_0  &\textrm{in} & \Omega \times\{0\}.
\ea
\right.
\eean
The  parabolic system \eqref{mainequation}  is   used to describe a wide variety of time-dependent phenomena, including heat conduction, particle diffusion, and pricing of derivative investment instruments. It is well  known that the system~\eqref{mainequation} has a unique solution $u(x, t) \in C^0\left([0, +\infty[; L^2(\Omega) \right) \cap C^0\left(]0, +\infty[; H^2(\Omega) \cap H^1_0(\Omega)\right)$ \cite{Ch09}.\\

The goal of this work is to study the following  inverse problem (P):  Given  $u(x, T) \in H^2( \Omega)$ for $T>0$,
 to find $\mathsf a \in \mathsf A$ such that  $u$ is a solution to the system~\eqref{mainequation}.\\

This inverse problem finds applications in multi-wave imaging and geophysics  \cite{AGKNS17, DY80, AV85, Ti35}.  It  can be seen as an extension to a non-stationary setting of a well  known  inverse elliptic problem with interior data, for which uniqueness and stability  have been already  derived \cite{A88, ADFV17, BCT19}. In such an elliptic context, it can be seen that  boundary information on the coefficient $a$ is needed, as well as  a unique continuation property of the gradient of solutions.  Notice that in dimension one a solution  of an inverse problem  similar to  (P)  was given  under some special assumptions on the boundary data   \cite{Is91}. The inverse problem (P) was recently cited among few other open inverse problems in \cite{A20}. Reviews for results concerning inverse problems for parabolic equations can be found in the following books \cite{Is06,Ch09,AK04}.\\

  In this paper we show that the inverse problem (P) has a unique solution, and we derive stability estimates  for the inversion  under  some  assumptions  on the point spectrum distribution of the associated elliptic operator,  the initial state $u_0$, and the observation 
time $T$.\\

It is well  known that the  unbounded  operator $ \L: L^2(\Omega)\rightarrow L^2(\Omega),$ defined by 
\[
\L:= -\div(\a\nabla \cdot),
\]
 with a Dirichlet boundary condition on $\Gamma$,  is self-adjoint, strictly  positive operator with a compact resolvent \cite{GT15}. Its domain is given by $D(\L)= H^1_0(\Omega)\cap H^2(\Omega)$. \\
 
 The  uniqueness and stability estimate presented here depend in an intricate way on the distribution of the eigenvalues of $\L$. 
 We denote by $\lambda_k,  k\in \mathbb N^*$,  the eigenvalues of  $\L$ arranged in a non-decreasing order and  repeated according to
 multiplicity. We also introduce the  strictly ordered  eigenvalues  $\hat \lambda_k,  k\in \mathbb N^*$. Notice that the first two values 
 of both sequences coincide.
 \begin{definition}
 We say that  $\L$   satisfies the property (G) with constants $\gamma \geq 0$ and  
 $\delta>0$ if its  eigenvalues $\lambda_k, k\in \mathbb N^*$, verify the following gap condition:  
  \bean \label{gap}
  \hat \lambda_{k+1} - \hat \lambda_k &\geq&  \delta \hat \lambda_{k}^{-\gamma}, \qquad k\in \mathbb N^*.
  \eean
\end{definition}
  \begin{remark}
  It is well  known that under  a non-trapping condition   
the operator  $\L$, subject to a  Dirichlet boundary condition  satisfies the property (G) 
 with $\gamma =0$, and $\delta>0$ is a constant  depending  only on $\a$ and $\Gamma$ \cite{BLR92, AT20}.  The property (G) is also
 somehow related to the boundary observability problem  for  Shr\"odinger and  wave equations in control theory 
 \cite{AT20, Zu07}. 
  \end{remark}
  The obtained stability estimate require that the property (G) be satisfied by the operator $\L$. Therefore for $\gamma \geq 0$, and 
  $\delta >0$ some fixed constants,  we introduce the set 
  \[
\mathsf A_0 = \left\{ a \in \A:  \L \textrm{ satisfies property  (G) with fixed  constants }  \gamma \geq 0  \textrm{ and }  \delta>0  \right\}.
\]
 
\begin{theorem}  \label{main} Let $\a, \tilde  \a \in \mathsf A_0$, and $d_\Omega (x)$ be the distance of $x$ to the boundary $\Gamma$.  Denote $u(x,t)$ and $\tilde u(x,t)$ the  solutions to the system~\eqref{mainequation}
with  respectively  diffusion coefficients $\a$ and $\tilde \a $. Assume that $ \int_{\Omega} u_0(x) d_\Omega(x) dx \not=0$. Then there exist $T_0>0$ and $C>0$   depending  on 
$\mathsf A_0$, $u_0$, $n$, and $\Omega$,  such that the following stability estimate holds
\bean \label{stability}
\| \a-\tilde \a \|_{L^2(\Omega)} \leq C e^{\a_+ \lambda_1^{\Omega} T} \|u(\cdot, T) - \tilde u (\cdot, T) \|_{H^2(\Omega)},
\eean  
for all $T>T_0$, where $\lambda_1^{\Omega} $ is the first eigenvalue of the Dirichlet Laplacian on $\Omega$. 

\end{theorem}

\begin{remark} The stability estimate  implies the uniqueness of the inverse problem (P).
The exponential growth of the Lipschitz stability constant~\eqref{stability} shows that the  inversion is in general  ill-posed. 
The exponential growth constant $\a_+ \lambda_1^{\Omega} $ can actually be replaced by $\min(\lambda_1,  \tilde \lambda_1)$, where
$\tilde \lambda_1$ is the first eigenvalue of $\tL$. The required regularity on the right hand side  of the stability estimate seems 
to be optimal, as for  the  inverse elliptic problem with interior data \cite{A88}. 

\end{remark}

The proof is based on a particular decomposition of  $-\partial_tu(x, T)$. The principal idea is to substitute $-\partial_tu(x, T)$
in  the parabolic  equation of the system~\eqref{mainequation} by $\hat \lambda_1 u(x, T) + \F(\a; x, T)$, where  
$\a\rightarrow \F(\a; x, T)$ is   a Lipschitz  non-linear function with Lipschitz constant  that decays  faster than $u(x, T)$ when 
$T$ tends towards infinity.  Moreover the function  $\F(\a; x, T)$ is independent of the data $u(x, T)$, and can be 
entirely recovered  from the knowledge $\a$, $u_0$, and $\Omega$.  In this regard, the unknown  coefficient $\a$  satisfies a nonlinear stationary transport equation
\bea
\div(\a \nabla u(x, T)) &=& - \hat \lambda_1u(x, T) +  \F(\a; x, T),  \quad  x\in \Omega.
\eea
Since $\F(\a; x, T)$ decays faster than $u(x, T)$ in  $L^2(\Omega)$ for large $T$,  the  system above can be considered as a nonlinear perturbation of a stationary linear transport equation 
\bea
\div(\a \nabla u(x, T)) &=& - \hat \lambda_1u(x, T),   \quad  x\in \Omega.
\eea
Consequently,  by solving the simplified linear equation above,  we shall be able to derive the  global stability estimate  for the
nonlinear one using classical  perturbation  methods. The detailed proof  is presented at the end of section 2. \\

 The paper is organized as follows. The first section is dedicated to some useful properties of the solution $u$ of the system
 \eqref{mainequation} including its  spectral decomposition. In section 2, we  provide the proof of the main Theorem~\ref{main}.
 In appendix \ref{App}, we recall some known useful properties of the eigenvalues and eigenfunctions of elliptic operators in a divergence form.
 
 \section{Preliminaries results} 
 We first derive  some  properties of the eigenelements of the unbounded operator $ L_a$. Considering $\ta$ as  a
 perturbation of the coefficient $\a$, we derive an upper bound of the perturbation of  eigenelements of  $ L_a$
 in terms of $\|\a-\ta\|_{L^2(\Omega)}$.

\begin{theorem} \label{interT1}
Let $\a, \ta \in \A$, and    $  (\lambda_k)_{  k\in \mathbb N^*} \subset \mathbb R^*$ (resp.  
$  (\tilde \lambda_k)_{  k\in \mathbb N^*} \subset \mathbb R^*$)  be respectively the  increasing  sequence of     eigenvalues
of $\L$ (resp. $\tL$). Then 
\bean \label{eigenIn}
|\lambda_k - \tilde \lambda_k|\leq C \min(\lambda_k, \tilde \lambda_k)^{1+\frac{n}{4}} \|\a-\ta\|_{L^2(\Omega)},
\eean
where $C>0$ is a constant that depends only on $n, \A$ and  $\Omega$. 
\end{theorem}
\begin{proof}
Without loss of generality we assume  that  $\lambda_k \geq  \tilde \lambda_k$.  \\

Denote  by $\phi_{k},  k \in \mathbb N^*$ (resp. 
$\tilde \phi_{k},  k \in \mathbb N^*$)  the orthonormal sequence of eigenfunctions of $\L$ (resp. $\tL$) associated to $\lambda_k, k\in \mathbb N^*$ (resp. $\lambda_k, k\in \mathbb N^*$). \\

Recall the Min-max characterization   of the eigenvalues $  (\tilde \lambda_k)_{  k\in \mathbb N^*}$  \cite{RS77}
\bean
\tilde \lambda_k  = \min_{\ba{ll}\Phi_k \subset H_0^1(\Omega)\\
\textrm{dim}(\Phi_k) = k\ea} \max_{ \phi \in \Phi_k \setminus\{0\}}
 \frac{\int_\Omega \ta |\nabla \phi|^2 dx}{ \int_{\Omega} |\phi|^2 dx}
\eean

In the expression above  the minimum is achieved  when $\Phi_k$  coincides with 
the finite dimension space generated by  $\widetilde V_k:= \{\tilde \phi_l: \; l\leq k\}$.   Therefore
\bea
\lambda_k -\tilde \lambda_k \leq \max_{ \phi \in \widetilde V_k \setminus\{0\}}
 \frac{\int_\Omega \a |\nabla \phi|^2 dx}{ \int_{\Omega} |\phi|^2 dx} -
 \max_{ \phi \in \widetilde V_k \setminus\{0\}}
 \frac{\int_\Omega \ta |\nabla \phi|^2 dx}{ \int_{\Omega} |\phi|^2 dx}.
\eea

Since  $\widetilde V_k$  is a finite dimension space the first maximum is reached at some vector  $\tilde \psi_k \in
 \widetilde V_k \setminus\{0\}$, satisfying $\int_{\Omega} |\tilde \psi_k|^2 dx =1$.  Hence 
\bean \label{fee}
\lambda_k -\tilde \lambda_k \leq 
 \int_\Omega \a |\nabla \tilde \psi_k|^2 dx -
  \int_\Omega \ta |\nabla \tilde \psi_k|^2 dx  \leq \|\a -\ta\|_{L^2(\Omega)} \left(\int_\Omega |\nabla \tilde \psi_k|^4 dx\right)^{\frac{1}{2}}.
\eean
Since  $\tilde \psi_k \in \widetilde V_k \setminus\{0\}$,  there exists a  real valued sequence $\alpha_l, l \leq k$, satisfying 
$ \sum_{l=1}^k \alpha_l^2 = 1$, and  $\tilde \psi_k= \sum_{l=1}^k \alpha_l \tilde \phi_l$.  Therefore,  $\tilde \psi_k$ verifies 
the following elliptic  equation \[\tL \tilde \psi_k = 
\sum_{l=1}^k \alpha_l \tilde \lambda_l \tilde \phi_l.\] We then 
 deduce from the classical  elliptic regularity  (Theorem 8.12  in \cite{GT15})
\bean \label{E1}
\|\tilde \psi_k\|_{H^2(\Omega)} \leq C_1\left\|\sum_{l=1}^k \alpha_l \tilde \lambda_l \tilde \phi_l\right\|_{L^2(\Omega)} \leq C_1 \tilde \lambda_k,
\eean
where $C_1>0$ depends only on $n,  \A$ and $\Omega$. \\

Classical Sobolev interpolation inequalities  for $n=2, 3$,  give \cite{GT15} 
\bean  \label{E2}
\|\nabla \tilde \psi_k\|_{L^4(\Omega)} \leq C_2  
\|\nabla \tilde \psi_k\|_{L^2(\Omega)}^{1-\frac{n}{4}}\|\tilde  \psi_k\|_{H^2(\Omega)}^{\frac{n}{4}}, 
\eean
where $C_2>0$ depends only on $n$ and $\Omega$. \\

 By a simple calculation, and using the fact that $\|\nabla \tilde \phi_l\|_{L^2(\Omega)} ^2 \leq  \tilde \lambda_l$, $l\in \mathbb N^*$,
  we obtain
\bean \label{E3}
\|\nabla \tilde \psi_k\|_{L^2(\Omega)} \leq  \tilde \lambda_k.
\eean

Combining inequalities \eqref{E1},  \eqref{E2} and \eqref{E3}, we get 
\bea 
\|\nabla \tilde \psi_k\|_{L^4(\Omega)} \leq C_3 \tilde \lambda_k^{\frac{n+4}{8}}, 
\eea
where $C_3>0$ depends only on $n, \A$ and $\Omega$.  We then deduce from \eqref {fee}
\bea
\lambda_k -\tilde \lambda_k \leq  C_3  \|\a -\ta\|_{L^2(\Omega)}  \tilde \lambda_k^{1+\frac{n}{4}},
\eea
which achieves the proof of the theorem.

\end{proof}
\begin{remark}
The estimate \eqref{eigenIn} may not be  optimal.  The objective here was to obtain an inequality with an uniform 
constant for all  functions $\a, \ta \in \A$.  
\end{remark}

\begin{theorem} \label{interT2}
Let $\a,  \ta \in \A_0$.  Let  $  P_k$ (resp.  
$  \widetilde P_k$)  be the orthogonal projection onto the eigenspace of  $\L$ (resp. $\tL$) corresponding  to the eigenvalue
$\hat \lambda_k$ (resp. $\widehat{\tilde \lambda_k}$). There exist constants $\eta>0$ and $C>0$  that
depend only on $n, \A_0$, and $\Omega$,
 such that if 
\bean \label{conda}
\|\a-\ta\|_{L^2(\Omega)} \leq \eta \max(\hat \lambda_k,  \widehat{\tilde \lambda}_k)^{-(1+\gamma+\frac{n}{4})},
\eean
then,  the  following estimate 
\bean \label{eigenfunctionIn}
\|P_k - \widetilde P_k\|_{ \mathcal{L}^2(\Omega) } \leq C 
(\max(\hat \lambda_k,  \widehat{\tilde \lambda}_k)^{\gamma+1}+1)^2  \|\a-\ta\|_{L^2(\Omega)},
\eean
holds.
\end{theorem}

\begin{proof}
In the proof $C>0$ denotes a generic constant depending $n$, $\A_0$, and $\Omega$. Without loss of generality 
we further assume that $\hat \lambda_k \geq  \widehat{\tilde \lambda}_k$.\\

Since $\a, \ta \in \A_0$, the gap condition \eqref{gap}  implies 
\bean \label{isolated}
B_{\rho_k}(\hat \lambda_k) \cap \{ \hat \lambda_l, \; l\in \mathbb N^*\} = \{\hat \lambda_k\},\qquad B_{\rho_k}(\widehat{ \tilde \lambda}_k) \cap \{(\widehat{  \tilde \lambda}_l, \; l\in \mathbb N^*\} = \{\widehat{\tilde \lambda}_k\},
\eean
where $B_{\rho_k}(z)$   is the complex disc of center 
$z \in \mathbb C$, and radius 
$\rho_k=  \frac{\delta}{4\hat \lambda_k^\gamma}$.  \\

On the other hand  estimate \eqref{eigenIn}  leads to
\bean \label{eigenIn2}
|\hat \lambda_k - \widehat{\tilde \lambda}_k|\leq C{\widehat{\tilde \lambda}}_k^{1+\frac{n}{4}} \|\a-\ta\|_{L^2(\Omega)}.
\eean

Now, combining inequalities \eqref{eigenIn2} and  \eqref{conda}, we have 

\bean \label{lambdaE1}
|\hat \lambda_k - \widehat{\tilde \lambda}_k|\leq C \eta \delta^{-1} \rho_k.
\eean

Choosing  $\eta>0$ small enough such that $C\eta < 1$, we  obtain 
\bean \label{lambdaE2}
\widehat{\tilde \lambda}_k \in  B_{\rho_k}(\hat \lambda_k).
\eean
Therefore, we also have 
\bean \label{isolated}
B_{\rho_k}(\hat{ \lambda}_k) \cap \{\widehat{  \tilde \lambda}_l, \; l\in \mathbb N^*\} = \{\widehat{\tilde \lambda}_k\}.
\eean
Consequently,  the resolvents  $
(\lambda I - \L)^{-1}$ and $ (\lambda I - \tL)^{-1}$
are well defined  as operators from $L^2(\Omega)$ onto $H^1_0(\Omega)\cap H^2(\Omega)$ for all $\lambda \in 
\partial B_{\rho_k}(\hat \lambda_k)$. In addition,
by the well-known Riesz formula,  we get \cite{Ka13}
\bea
P_k = -\frac{1}{2\iota \pi}\int_{|\lambda -\hat \lambda_k|=\rho_k} (\lambda I - \L)^{-1} d\lambda,\qquad 
\widetilde P_k = -\frac{1}{2\iota \pi}\int_{|\lambda -\hat {\lambda}_k|=\rho_k} (\lambda I - \L)^{-1} d\lambda,
\eea
where $\iota \in \mathbb C$ is the imaginary complex  number, and $I$ is the identity operator. \\

Hence
\bea
P_k - \widetilde P_k = \frac{1}{2\iota \pi}\int_{|\lambda -\hat {\lambda}_k|=\rho_k} (\lambda I - \L)^{-1}(\L-\tL)  (\lambda I - \tL)^{-1}d\lambda.
\eea
Since $P_k$ and $\widetilde P_k $ are orthogonal projections, $P_k - \widetilde P_k$  is  a self-adjoint  bounded operator  from $L^2(\Omega)$ to itself.\\

 Consequently
\bean \label{projecE1}
\|P_k - \widetilde P_k\|_{ \mathcal{L}^2(\Omega) } &=& (2\pi)^{-1} \sup_{\lambda \in \partial B_{\rho_k}(\hat \lambda_k),\,  f\in W _k,\, \|f\|_{L^2(\Omega)} = 1} 
|\left \la (\L-\tL)\tilde u_f^\lambda,  u_f^\lambda \right \ra_{L^2(\Omega)}|, \nonumber \\
 &=& (2\pi)^{-1}  \sup_{\lambda \in \partial B_{\rho_k}(\hat \lambda_k),\,  f\in W _k,\, \|f\|_{L^2(\Omega)} = 1} 
 \left| \int_{\Omega} (\a-\ta) \nabla u_f^\lambda \nabla \tilde u_f^\lambda dx \right|, \nonumber \\
 &\leq & (2\pi)^{-1}  \| \a-\ta \|_{L^2(\Omega)}  \sup_{\lambda \in \partial B_{\rho_k}(\hat \lambda_k),\,  f\in W _k,\, \|f\|_{L^2(\Omega)} = 1} 
\|\nabla u_f^\lambda\|_{L^4(\Omega)}\|\nabla \tilde u_f^\lambda\|_{L^4(\Omega)},
\eean
where $u_f^\lambda=  (\lambda I - \L)^{-1}f, \; \tilde u_f^\lambda =  (\lambda I - \tL)^{-1}f,$ and 
$W_k$ is the finite dimension   vector space in $L^2(\Omega)$,  spanned by the eigenfunctions associated to $\hat \lambda_k$,
and $\widehat{\tilde  \lambda}_k$.\\

By construction, we have 
\bean \label{ufE1}
\|u_f^\lambda\|_{L^2(\Omega)},  \|\tilde u_f^\lambda\|_{L^2(\Omega)} \leq \frac{1}{\rho_k}
\eean
Similar to the proof of Theorem \ref{interT1},  we deduce from the classical  elliptic regularity  
\bea
\| u_f^\lambda \|_{H^2(\Omega)} \leq  C (\lambda\|u_f^\lambda\|_{L^2(\Omega)}+1),\quad
\|\tilde u_f^\lambda \|_{H^2(\Omega)} \leq  C (\lambda\|u_f^\lambda\|_{L^2(\Omega)}+1),
\eea
which associated to  inequalities  \eqref{ufE1},  provide 
\bean\label{ufE2}
\|u_f^\lambda \|_{H^2(\Omega)}, \, \|\tilde u_f^\lambda \|_{H^2(\Omega)}\leq  C(\frac{\lambda }{\rho_k}+1).
\eean

Sobolev embedding  Theorem  gives \cite{GT15} 
\bean  \label{ufE3}
\|\nabla  u_f^\lambda \|_{L^4(\Omega)} \leq C  \| u_f^\lambda \|_{H^2(\Omega)}, \quad 
\|  \nabla \tilde u_f^\lambda \|_{L^4(\Omega)} \leq C 
\| \tilde u_f^\lambda \|_{H^2(\Omega)}.
\eean
Combining estimates \eqref{ufE2} and \eqref{ufE3}, we finally obtain 
\bean \label{ufE4}
\|\nabla  u_f^\lambda \|_{L^4(\Omega)}, \, \|\nabla  \tilde u_f^\lambda \|_{L^4(\Omega)} \leq C 
(\frac{\lambda}{\rho_k}+1).
\eean
We infer from \eqref{projecE1} 
\bea
\|P_k - \widetilde P_k\|_{ \mathcal{L}^2(\Omega) } &\leq &   C (\hat \lambda_k^{\gamma+1}+1)^2 \| \a-\ta \|_{L^2(\Omega)}.
\eea

\end{proof}
\section{Proof of Theorem~\ref{main}}
We first introduce the nonlinear function $\F(\a; x, T)$, and show  that its Lipschitz continuous modulus with respect 
to $\a$,  decays faster than  $u(x, T)$ for large $T$. Without loss of generality we further
 assume that   $\int_{\Omega} u_0(x) d_{\Omega}(x) dx > 0$.\\

Define for  $\a \in \A_0$ and $T>0$, the nonlinear  function $\F(\a; x, T) \in L^2(\Omega)$,  by 
\bean \label{functionF}
\F(\a; x, T) &=& \partial_tu(x, T)+ \hat \lambda_1 u(x, T), \qquad x\in \Omega,
\eean
where $u$ is the unique solution of the system  ~\eqref{mainequation}. 
\begin{theorem} \label{estimF}
Let  $\a, \ta \in \A_0$. Then there exists a constant $C>0$  that depends only on $\theta, \Omega, n, u_0$ and 
$\A_0$, such that the inequality 
\bean \label{estimFeq}
\| \F(\a; x, T) - \F(\ta; x, T)\|_{L^2(\Omega)} \leq C e^{-\min(\hat \lambda_2, \widehat{\tilde \lambda}_2)T} \|\a -\ta\|_{L^2(\Omega)},
\eean
is valid for all $T\geq 1$.
\end{theorem}
\begin{proof} In the proof $C>0$ denotes a generic constant depending  on $n$, $\A_0$, $u_0$, and  $\Omega$.\\

We start the proof by writing the decomposition of  $\F(\a; x, T)$ (resp. $\F(\ta; x, T)$) in terms  of the eigenfunctions of the elliptic operator 
$\L$ (resp. $\tL$). Recall  $  P_k$ (resp.   $  \widetilde P_k$)   the orthogonal projection onto the eigenspace of  $\L$ (resp. $\tL$) associated   to the eigenvalue $\hat \lambda_k$ (resp. $\widehat{\tilde \lambda}_k $). \\

It is well known that $u$ and $\tilde u$ have the following spectral  decomposition \cite{Ch09}
\bea
u(x, t) \;=\; \sum_{k=1}^\infty e^{-\hat \lambda_k t} P_k u_0(x); \qquad  \tilde 
u(x, t) \;=\; \sum_{k=1}^\infty e^{-\widehat{\tilde \lambda}_k t} \widetilde P_k u_0(x)
\eea

Forward calculations yield 
\bea
\F(\a; x, T) = \sum_{k=2}^\infty  (\hat \lambda_k - \hat \lambda_1)e^{-\hat \lambda_k T} P_k u_0(x); \qquad  \tilde 
 \F(\a; x, T) = \sum_{k=2}^\infty(\widehat{\tilde \lambda_k} - \widehat{\tilde \lambda_1}
 ) e^{-\widehat{\tilde \lambda}_k T} \widetilde P_k u_0(x).
\eea
Hence 
\bea
\F(\a; x, T) - \F(\ta; x, T)= \\
 \sum_{k=2}^\infty \left[ (\hat \lambda_k - \hat \lambda_1)e^{-\hat \lambda_k T}
-  (\widehat{\tilde \lambda_k} - \widehat{\tilde \lambda_1}
 ) e^{-\widehat{\tilde \lambda}_k T} \right] \widetilde P_k u_0(x)+
 \sum_{k=2}^\infty (\hat{ \lambda_k} - \hat{\lambda_1}
 ) e^{-\hat{ \lambda}_k T}   \left[P_k -\widetilde P_k\right]u_0(x)\\= \F_1+\F_2.
\eea

Let $\beta_k, k\in \mathbb N\setminus \{0, 1\}, $ defined by
\bea
\beta_k = \min(\hat  \lambda_k, \widehat{\tilde \lambda_k}).
\eea
Then 
\bea
\| \F_1 \|_{L^2(\Omega)}^2 \leq   \sum_{k=2}^\infty \left[ |\hat  \lambda_k  - \widehat{\tilde \lambda}_k|(1+\beta_kT) +
|\hat  \lambda_1  - \widehat{\tilde \lambda}_1| \right]^2 e^{-2\beta_kT} \|u_0\|_{L^2(\Omega)}^2. 
\eea
Using results of Theorem \ref{interT1} and Lemma~\ref{lemA2}, we obtain 

\bea
\| \F_1 \|_{L^2(\Omega)}^2  \hspace{-2mm}&\leq& \hspace{-2mm} C  \sum_{k=2}^\infty \left[ \beta_k^{1+\frac{n}{4}}(1+\beta_kT) +
\beta_1^{1+\frac{n}{4}}\right]^2 e^{-2\beta_kT}\|u_0\|_{L^2(\Omega)}^2 \|\a -\ta\|_{L^2(\Omega)}^2,\\
 \hspace{-2mm}&\leq& \hspace{-2mm} C  \sum_{k=2}^\infty \beta_k^{2+\frac{n}{2}} e^{-\beta_kT}\|u_0\|_{L^2(\Omega)}^2 \|\a -\ta\|_{L^2(\Omega)}^2  \leq Ce^{-\beta_2T}\sum_{k=2}^\infty \beta_k^{2+\frac{n}{2}} e^{-(\beta_k-\beta_2)}\|u_0\|_{L^2(\Omega)}^2 \|\a -\ta\|_{L^2(\Omega)}^2.
\eea
Note that  by Weyl's asymptotic formula, we have $\lambda_k \sim C k^{\frac{2}{n}}$ for large $k$, which 
guarantees the convergence of the series above \cite{RS77, He06}.\\

The results of Lemma~\ref{lemA3} imply 
\bean \label{first}
\| \F_1 \|_{L^2(\Omega)} &\leq& Ce^{- \beta_2 T}\|\a -\ta\|_{L^2(\Omega)}.
\eean
 Since the orthogonality  of the  terms of the series $\F_2$ is no longer true, and the 
 fact that the perturbation does not affect uniformly the eigenfunctions,  deriving an upper
 bound for  $\| \F_2 \|_{L^2(\Omega)}$ is more involved. \\
 
 Recall that the sequences $\hat \lambda_k$  and $\widehat{\tilde \lambda}_k$  are strictly increasing. Let 
 $N\in \mathbb N^*$ be the smallest integer satisfying 
 \bean \label{NN}
\|\a-\ta\|_{L^2(\Omega)} >  \eta \max(\hat \lambda_k,  \widehat{\tilde \lambda}_k)^{-(1+\gamma+\frac{n}{4})}
\geq \eta \beta_k^{-(1+\gamma+\frac{n}{4})},\qquad \forall
k\geq N,
\eean
 where $\eta>0$ is the  constant  introduced in Theorem \ref{interT2}.\\
 
Next, we split  $\F_2$ into two parts: 
 
 \bean \label{split}
 \F_2= \sum_{k=2}^{N-1} (\hat{ \lambda}_k - \hat{\lambda}_1
 ) e^{-\hat{ \lambda}_k T}   \left[P_k -\widetilde P_k\right]u_0(x) + \sum_{k=N}^\infty (\hat{ \lambda}_k - \hat{\lambda}_1
 ) e^{-\hat{ \lambda}_k T}   \left[P_k -\widetilde P_k\right]u_0(x) = \F_{21} + \F_{22},
 \eean
  with the convention that the  first sum $F_{21} = 0$ when $N=2$.\\

  We deduce from \eqref{NN} the following estimate
  \bea
  \| \F_{22} \|_{L^2(\Omega)}^2 &\leq&\eta^{-2} \sum_{k=N}^\infty (\hat{ \lambda}_k - \hat{\lambda}_1
 )^2 \beta_k^{2(1+\gamma+\frac{n}{4})} e^{-2 \beta_k T} 
 \|u_0\|_{L^2(\Omega)}^2\|\a-\ta\|_{L^2(\Omega)}^2,\\
 &\leq &\eta^{-2}  e^{-2 \beta_k T} \sum_{k=2}^\infty (\hat{ \lambda}_k - \hat{\lambda}_1
 )^2 \beta_2^{2(1+\gamma+\frac{n}{4})} e^{-2 (\beta_k-\beta_2)}  \|u_0\|_{L^2(\Omega)}^2\|\a-\ta\|_{L^2(\Omega)}^2.
  \eea
  Again using the  upper and lower bounds  derived in Lemma~\ref{lemA3}, we obtain
  
 \bean \label{SS1}
  \| \F_{22} \|_{L^2(\Omega)} &\leq& Ce^{- \beta_2 T}\|\a -\ta\|_{L^2(\Omega)}.
  \eean
  
 On the other hand, we have 
 \bea
 \| \F_{21} \|_{L^2(\Omega)}^2 &\leq& \left\|\sum_{k=2}^{N-1} (\hat{ \lambda}_k - \hat{\lambda}_1
 ) e^{-\hat{ \lambda}_k T}   \left|\left[P_k -\widetilde P_k\right]u_0\right| \right\|_{L^2(\Omega)}^2. 
 \eea
 Cauchy-Shwartz inequality gives
  \bea
 \| \F_{21} \|_{L^2(\Omega)}^2 &\leq& \left(\sum_{k=2}^{N-1} (\hat{ \lambda}_k - \hat{\lambda}_1
 )^2 e^{-\hat{ \lambda}_k T} \right)\left(  \sum_{k=2}^{N-1} e^{-\hat{ \lambda}_k T} \left\|P_k -\widetilde P_k \right\|_{\mathcal L(L^2(\Omega))}^2\right)
  \|u_0\|_{L^2(\Omega)}^2. 
 \eea
By construction, we have  
\bean \label{NN2}
\|\a-\ta\|_{L^2(\Omega)} \leq   \eta \max(\hat \lambda_k,  \widehat{\tilde \lambda}_k)^{-(1+\gamma+\frac{n}{4})},\qquad \forall 
k\leq N-1.
\eean
Using the results  of Theorem~\ref{interT2},  leads to
 \bea
 \| \F_{21} \|_{L^2(\Omega)}^2\\  \leq C\left(\sum_{k=2}^{N-1} (\hat{ \lambda}_k - \hat{\lambda}_1
 )^2 e^{-\hat{ \lambda}_k T}\right)\left(   \sum_{k=2}^{N-1} (\max(\hat \lambda_k,  \widehat{\tilde \lambda}_k)^{\gamma+1}+1)^4
 e^{-\hat{ \lambda}_k T}\right)
  \|u_0\|_{L^2(\Omega)}^2  \|\a-\ta\|_{L^2(\Omega)}^2, \\
  \leq  Ce^{-2  \beta_2 T} \left( \sum_{k=2}^{\infty} (\hat{ \lambda}_k - \hat{\lambda}_1
 )^2 e^{-(\hat{ \lambda}_k - \hat{ \lambda}_2)} \right) \left( \sum_{k=2}^{\infty} (\max(\hat \lambda_k,  \widehat{\tilde \lambda}_k)^{\gamma+1}+1)^4
 e^{-(\hat{ \lambda}_k- \hat \lambda_2)}\right)
  \|u_0\|_{L^2(\Omega)}^2  \|\a-\ta\|_{L^2(\Omega)}^2.
 \eea
 Applying again the bounds in Lemma~\ref{lemA3}, we get
 \bean 
 \label{SS2}
  \| \F_{21} \|_{L^2(\Omega)} &\leq& Ce^{- \beta_2 T}\|\a -\ta\|_{L^2(\Omega)}.
 \eean
Finally,  combining inequalities \eqref{first}, \eqref{SS1}, and \eqref{SS2}, conducts  to the desired estimate.
\end{proof}
We  next study the decay behavior of $\partial_t u(x, T)$ and $|\nabla u(x, T)|^2$ as $T$ tends towards infinity.
\begin{theorem}
Let  $\a \in  \A_0$, and $u$ be the unique solution to the system \eqref{mainequation}. Then there exist $T_1>0$, $\varepsilon_0>0$,
and 
$C>0$   depending only 
on $\A_0$, $\Omega, n $ and $u_0$ such that the following inequalities 
\bean \label{NN0}  u(x, T) &\geq& C e^{-\lambda_1 T}\phi_1(x), \qquad \hspace{8mm} \forall x\in \Omega,\\ 
\label{NN1}
-\partial_tu(x, T) &\geq& C e^{-\lambda_1 T}\phi_1(x), \qquad \hspace{8mm}  \forall x\in \Omega,\\ \label{NN2}
|\nabla u(x, T)|^2 &\geq& C e^{-2\lambda_1 T} |\nabla \phi_1(x)|^2, \qquad \forall x\in \Omega_\varepsilon,
\eean 
hold for all $T\geq T_1$, and $0<\varepsilon< \varepsilon_0$,  with  $\Omega_\varepsilon= \{x\in \Omega:  \; d_{\Omega}(x) < \varepsilon \}$.

\end{theorem}
\begin{proof}
In the sequel $C>0$ denotes a generic constant that depends  only on $\A_0$, $\Omega, n $ and $u_0$.  We further assume 
that $T\geq 1$.\\

The proof is based on the following decomposition of $u $ in terms of the eigenfunctions of $\L$:
\bea
u(x, t) &=& \sum_{k=1}^\infty e^{-\hat \lambda_k t} P_k u_0(x), \qquad \forall t>0, \; x\in \Omega. 
\eea
For $T\geq 1$, we have 
\bea
\nabla u(x, T) &= &\sum_{k=1}^\infty e^{-\hat \lambda_k T} \nabla P_k u_0(x), \qquad \forall  x\in \Omega,\\
-\partial_t  u(x, T)  &= &\sum_{k=1}^\infty \lambda_k e^{-\hat \lambda_k T} \nabla P_k u_0(x), \qquad \forall  x\in \Omega.
 \eea
 Therefore 
 \bean
-\partial_t u(x, T)  \hspace{-3mm}  &\geq&  \hspace{-3mm}  \hat \lambda_1 e^{- \lambda_1 T}  |P_1 u_0(x)| -  \sum_{k=2}^\infty \hat \lambda_k e^{-\hat \lambda_k T} |P_k u_0(x)|, \\
|\nabla u(x, T)|^2  \hspace{-3mm}  &\geq & \hspace{-3mm}  \frac{\lambda_1^2}{2} e^{- 2\lambda_1 T} |\nabla P_1 u_0(x)|^2-  
\left(\sum_{k=2}^\infty e^{-\hat \lambda_k T} \left|\nabla P_k u_0(x))\right|^2)\right)^{\frac{1}{2}}
\left(\sum_{k=2}^\infty e^{-\hat \lambda_k T}\right)^{\frac{1}{2}},  \label{gr1}
 \eean
 for all  $x\in \Omega$.\\
 
 Next,  we derive the first inequality of the lemma.  Using inequalities \eqref{ineqEiG},  we obtain 
 
 \bea
 -\partial_t u(x, T) &\geq  &  \left[  \hat \lambda_1 e^{- \hat \lambda_1 T} \|P_1 u_0(x)\|_{L^2(\Omega)} 
-  \sum_{k=2}^\infty \hat \lambda_k^{\frac{3}{2} +\frac{n}{4}} e^{-\hat \lambda_k T} \|u_0\|_{L^2(\Omega)} 
\right]\phi_1(x),\\
&\geq  &  \left[  \hat \lambda_1e^{- \hat \lambda_1 T} \| P_1 u_0(x)\|_{L^2(\Omega)} 
-  e^{- \hat \lambda_2 T}  \sum_{k=2}^\infty \hat \lambda_k^{\frac{3}{2} +\frac{n}{4}} e^{-(\hat \lambda_k- \hat \lambda_2) } \|u_0\|_{L^2(\Omega)} 
\right]\phi_1(x).
 \eea
 Since $a\in \A_0$, the following  gap condition~\eqref{gap}, holds
 \bea
 \hat \lambda_2 - \hat \lambda_1 \geq \frac{\delta}{\hat \lambda_1}. 
 \eea
 Notice that the gap condition between the two first eigenvalues  is always fulfilled \cite{He06}.\\
 
  We deduce from Lemma~\ref{lemA3}
 \bean \label{gap2}
 \hat \lambda_2 - \hat \lambda_1 \geq \frac{\delta}{a_+ \lambda_1^{\Omega}}. 
 \eean
 Hence
 \bea
 \partial_t u(x, T)  &\geq  & e^{- \hat \lambda_1 T} \left[\hat \lambda_1  \| P_1 u_0(x)\|_{L^2(\Omega)} 
-  e^{-  \frac{\delta}{a_+ \lambda_1^{\Omega}}T}  \sum_{k=2}^\infty \hat \lambda_k^{\frac{3}{2} +\frac{n}{4}} e^{-(\hat \lambda_k- \hat \lambda_2) } \|u_0\|_{L^2(\Omega)} 
\right]\phi_1(x).
 \eea
 Again using Lemma ~\ref{lemA3} leads to 
  \bea
 |\partial_t u(x, T)|  &\geq  & e^{- \hat \lambda_1 T} \left[\lambda_1^\Omega \| P_1 u_0(x)\|_{L^2(\Omega)} 
-  Ce^{-  \frac{\delta}{a_+ \lambda_1^{\Omega}}T}  \|u_0\|_{L^2(\Omega)} 
\right]\phi_1(x).
 \eea
 
 On the other hand,  we deduce from Lemma \ref{lemA2} 
 \bea
  \| P_1 u_0(x)\|_{L^2(\Omega)} = \left| \int_{\Omega} u_0(x) \phi_1(x) dx \right|
  \geq C  \int_{\Omega} u_0(x) d_{\Omega}(x) dx.
 \eea
Since $\int_{\Omega} u_0(x) d_{\Omega}(x) dx > 0 $, we  have $ \| P_1 u_0(x)\|_{L^2(\Omega)} =
\int_{\Omega} u_0(x) d_{\Omega}(x) dx$. Hence there exists a 
unique   $ T_{11} \in \mathbb R$  solution to  the equation 
\bea
 \| P_1 u_0(x)\|_{L^2(\Omega)} 
-  Ce^{-  \frac{\delta}{a_+ \lambda_1^{\Omega}}T_{11}}  \|u_0\|_{L^2(\Omega)}  =0.
\eea
Obviously $T_{11}$ depends only on $\A_0$, $\Omega, n $ and $u_0$. \\

Consequently 
\bea
 -\partial_tu(x, T)  &\geq  & C e^{- \hat \lambda_1 T} \phi_1(x),\qquad \forall x\in \Omega, \; t \in ]T_{11}, +\infty[.
 \eea
Similar analysis on  $u(x, T)$, leads to 
 \bea
 u(x, T)  &\geq  & C e^{- \hat \lambda_1 T} \phi_1(x),\qquad \forall x\in \Omega, \; t \in ]T_{12}, +\infty[,
 \eea
where $T_{12} \in \mathbb R$.\\

Now, we shall focus on the second inequality. To do so we need to estimate $\|\nabla \phi_k\|_{L^\infty(\Omega)}$. There 
are many works dealing with optimal increasing rate of $\| \phi_k\|_{L^\infty(\Omega)}$ and $\|\nabla \phi_k\|_{L^\infty(\Omega)}$
in terms of $\lambda_k$, when $k$ tends to infinity (see for instance \cite{So88, Xu} and references therein). Most existing results 
considered 
the  Laplacian operator or  did not  pay attention to the regularity of the elliptic coefficients. Since the optimal  increasing rate 
is out of the focus of this work,  we prefer here deriving similar estimates using  classical elliptic  regularity combined with 
results of Lemma \ref{lemA4}. \\

We deduce from elliptic regularity (Theorem 9.12 in \cite{GT15})
\bea
\| \phi_k\|_{W^{2, n+\frac{1}{n}}(\Omega)} \leq C (1+\hat \lambda_k) \|\phi_k\|_{ L^{n+\frac{1}{n} }(\Omega) }.
\eea
 By Sobolev embedding Theorem, we have
 \bea
 \| \nabla  \phi_k\|_{C^{\frac{1}{n^2+1}}(\overline \Omega)} \leq C (1+\hat \lambda_k) \|\phi_k\|_{L^{\infty}(\Omega)}.
 \eea
Lemma~ \ref{lemA4} yields 

 \bean \label{gradE}
 \| \nabla  \phi_k\|_{C^{\frac{1}{n^2+1}}(\overline \Omega)} \leq C (1+ \hat \lambda_k) \hat \lambda_k^{\frac{1}{2}+\frac{n}{4}}.
 \eean
 
Combining \eqref{gradE} with \eqref{gr1}, we get 

 \bea
 |\nabla u(x, T)|^2  \geq \\
   \frac{\hat \lambda_1^2}{2} e^{- 2 \hat \lambda_1 T}  \| P_1 u_0(x)\|_{L^2(\Omega)}   |\nabla \phi_1(x)|^2-  
\left(\sum_{k=2}^\infty (1+\hat \lambda_k)^2 \hat\lambda_k^{1+\frac{n}{2}}e^{-2\hat \lambda_k T}\|u_0\|_{L^2(\Omega)}^2\right)^{\frac{1}{2}} 
\left(\sum_{k=2}^\infty e^{-\hat \lambda_k T}\right)^{\frac{1}{2}}\hspace{-1mm}, 
 \eea
 for all $ x\in \Omega$.\\

Consequently 
 \bea
 |\nabla u(x, T)|^2 \geq \\
\frac{ \hat \lambda_1^2}{2} e^{- 2 \hat \lambda_1 T}  \| P_1 u_0(x)\|_{L^2(\Omega)}   |\nabla \phi_1(x)|^2\hspace{-1mm}-  \hspace{-1mm}
e^{- 2 \hat \lambda_2 T} \left(\sum_{k=2}^\infty (1+\hat \lambda_k)^2 \hat \lambda_k^{1+\frac{n}{2}}e^{-2(\hat \lambda_k- \hat \lambda_2)}
\|u_0\|_{L^2(\Omega)}^2\right)
^{\frac{1}{2}}\hspace{-2mm}
\left(\sum_{k=2}^\infty e^{-(\hat \lambda_k - \hat \lambda_2)}\right)^{\frac{1}{2}}\hspace{-1mm}, 
 \eea
 for all $ x\in \Omega$, and $T\geq 1$.\\
 
 We deduce from Lemma \ref{lemA3} and inequality \eqref{gap2}, the following estimate 
 \bean \label{in2T} \hspace{6mm}
  |\nabla u(x, T)|^2 \; \geq \; e^{- 2 \hat \lambda_1 T}  \left[ \frac{ (\lambda_1^\Omega)^2}{2}  \| P_1 u_0(x)\|_{L^2(\Omega)}   |\nabla \phi_1(x)|^2-  
Ce^{-  \frac{\delta}{2a_+ \lambda_1^{\Omega}}T}  \|u_0\|_{L^2(\Omega)}\right] , \quad \forall  x\in \Omega, \, T\geq 1.
 \eean

\begin{proposition} \label{1eigen}
There exist  constants $\varepsilon_0>0$ and $C_0>0$ that  depend only  on $\Omega, n$, and $\A_0$, such that 
the following inequality 
\bean \label{invar}
|\nabla \phi_1(x)| &\geq& C_0, \qquad \forall x\in \Omega_\varepsilon,
\eean
holds for all $0<\varepsilon< \varepsilon_0$.
 \end{proposition}
 \begin{proof}
 For $ \varepsilon>0 $ small enough, $\Omega_\varepsilon$ becomes a tubular domain, and can be parametrized by
 \bea
 \Omega_\varepsilon= \{x+ s\nu(x):  \; \; x\in \Gamma, \;  0<s <  \varepsilon\}.
 \eea
 We deduce from \eqref{gradE}, the following estimate
 \bea
 \left|-\nabla \phi_1(x+ s\nu(x))\cdot \nu(x) +\partial_{\nu} \phi_1(x)\right| \leq C 
 (1+ \hat \lambda_1) \hat \lambda_1^{\frac{1}{2}+\frac{n}{4}} \varepsilon^{\frac{1}{n^2+1}}, \qquad \forall s\in ]0,  \varepsilon[, \; x \in 
 \Gamma.
 \eea
Hence 
 \bea
 | \nabla \phi_1(x+ s\nu(x))| \geq  -\nabla \phi_1(x+ s\nu(x)) \cdot \nu(x)| \geq  -\partial_{\nu} \phi_1(x) -  C (1+ \hat \lambda_1) \hat \lambda_1^{\frac{1}{2}+\frac{n}{4}} 
  \varepsilon^{\frac{1}{n^2+1}}, \quad \forall s\in ]0,  \varepsilon[, \; x \in 
 \Gamma.
 \eea
 Recall that  the constants $\varepsilon_0>0$ and $C_0$ in inequality \eqref{invar}   depend only 
on  $\Omega, n$, and $\A_0$.
Since  inequality \eqref{invar}  is valid for all $\varepsilon \in ]0, \varepsilon_0[$, we 
deduce from  Lemmata \ref{lemA2}, and  \ref{lemA3},  the desired result.
\end{proof}
Let $T_{13} \in \mathbb R$ be the unique solution to  the following equation 
\bea
 \frac{1}{2}\| P_1 u_0(x)\|_{L^2(\Omega)} C_0= 
Ce^{-  \frac{\delta}{2a_+ \lambda_1^{\Omega}}T_{13}}  \|u_0\|_{L^2(\Omega)},
\eea
where $C_0>0$ is the constant of inequality \eqref{invar}. \\

Applying the results of Proposition~\eqref{1eigen} to the inequality  \eqref{in2T},  we obtain
\bean \label{in3T}
  |\nabla u(x, T)|^2 & \geq & e^{- 2 \hat \lambda_1 T}  \frac{ (\lambda_1^\Omega)^2}{4}  \| P_1 u_0\|_{L^2(\Omega)} |\nabla \phi_1(x)|^2,
   \eean 
for all $T\geq T_{12}.$\\

By taking $T_1 = \max(1, T_{11}, T_{12}, T_{1,3})$, we achieve the proof of the theorem.
\end{proof}
Now, we are ready to prove the main Theorem~\ref{main}. Without loss of generality, we assume 
that $\hat \lambda_2= \min(\hat \lambda_2, 
\widehat {\tilde \lambda_2}).$ Further  $C>0$ denotes a generic constant that depending 
on $\A_0$, $\Omega, n $ and $u_0$.\\

Since $u$ satisfies \eqref{mainequation}, $\a$ verifies the following nonlinear transport equation
\bean \label{transp1}
\div(\frac{\a}{\hat \lambda_1}\nabla u(x, T)) = - u(x, T) + \frac{1}{ \hat \lambda_1} \F(\a; x, T),  \quad  x\in \Omega.
\eean
Similarly, $\ta$  satisfies the following nonlinear transport equation
\bean \label{transp2}
\div(\frac{\ta}{\widehat{\tilde \lambda_1}} \nabla \tilde u(x, T)) = - \tilde u(x, T) +\frac{1}{\widehat{\tilde  \lambda}_1}\F(\ta; x, T), \quad  x\in \Omega.
\eean

 Taking the difference between the two equations \eqref{transp1}  and \eqref{transp2}, we get 
 \bean \label{aa}
 \div((\frac{\a}{\hat \lambda_1}-\frac{\ta}{\widehat{\tilde \lambda}}) \nabla  u(x, T)) =\\  \div( 
 \frac{\ta}{\widehat{\tilde \lambda}_1}\nabla ( u(x, T) -\tilde u(x, T)))  +
   \tilde u(x, T) - u(x, T) +  \frac{1}{ \hat  \lambda_1} \F(\a; x, T)- \frac{1}{\widehat{\tilde \lambda}_1}\F(\ta; x, T), 
   \quad  x\in \Omega. \nonumber
 \eean
 \begin{proposition} \label{mainProp}
 There exist  constants $T_2>0$, $C>0$ and $\varepsilon_0>0$ that  depend 
 only on $\Omega, n,\,$   $\A_0$ and $u_0$ such that  the following 
 inequalities 
 
 \bean
 \label{XX1}
 |\frac{1}{\hat \lambda_1}-\frac{1}{\widehat{\tilde \lambda_1}}| &\leq& C \left[ e^{\hat \lambda_1 T }
  \|u-\tilde u \|_{L^2(\Omega)}+e^{-(\hat \lambda_2 - \hat \lambda_1)T}\|\a-\ta\|_{L^2(\Omega)}\right],
 \eean
and 
 \bean \label{XX2}
 \int_{\Omega_\varepsilon}| \frac{\a}{\hat \lambda_1}-\frac{\ta}{\widehat{\tilde \lambda_1}}|^2 | \nabla \phi_1 |^2  dx+
 \int_{\Omega}| \frac{\a}{\hat \lambda_1}-\frac{\ta}{\widehat{\tilde \lambda_1}}|^2 \phi_1^2dx\\ 
 \leq C\left[e^{\hat \lambda_1 T }\|u-\tilde u \|_{H^2(\Omega)}
 +e^{-(\hat \lambda_2 - \hat \lambda_1)T}\left(\|\a-\ta\|_{L^2(\Omega)}+   
 |\frac{1}{\hat \lambda_1}-\frac{1}{\widehat{\tilde \lambda_1}}| \right)
 \right]\left\|  \frac{\a}{\hat \lambda_1}-\frac{\ta}{\widehat{\tilde \lambda_1}}\right\|_{L^2(\Omega)},\nonumber
 \eean
 hold for all $T\geq T_2$, and $\varepsilon \in ]0, \varepsilon_0[$.
 
 \end{proposition}
 \begin{proof}
 Multiplying  equation \eqref{aa} by $1$, and integrating over $\Omega$, we obtain
 \bean \label{MM1}
|\frac{1}{\hat \lambda_1}-\frac{1}{\widehat{\tilde \lambda_1}}|\left| \int_\Gamma \a_0 \partial_\nu u(., T) ds(x)\right|\\\leq  
C \left[\|u-\tilde u \|_{L^2(\Omega)} + |\frac{1}{\hat \lambda_1}-\frac{1}{\widehat{\tilde \lambda_1}}| \|\F(\a; ., T) \|_{L^2(\Omega)}
+ \frac{1}{\widehat{\tilde \lambda_1}} \|\F(\a; ., T) - \F(\ta; ., T)\|_{L^2(\Omega)}  \right]. \nonumber
 \eean
 Recall 
 \bea
 \|\F(\a; ., T) \|_{L^2(\Omega)}^2 &=& \sum_{k=2}^\infty  (\hat \lambda_k - \hat \lambda_1)^2e^{-2\hat \lambda_k T} \|P_k u_0
 \|_{L^2(\Omega)}^2,\\
 &\leq & \sum_{k=2}^\infty  (\hat \lambda_k - \hat \lambda_1)^2e^{-2\hat \lambda_k T} \|u_0\|_{L^2(\Omega)}^2 \\
  &\leq & 
   \sum_{k=2}^\infty  (\hat \lambda_k - \hat \lambda_1)^2e^{-2(\hat \lambda_k-  \hat \lambda_2) } \|u_0\|_{L^2(\Omega)}^2e^{-2 \hat \lambda_2 T},
 \eea
 for all $T\geq 1$.\\
 
Consequently
 \bean \label{FF1}
 \|\F(\a; ., T) \|_{L^2(\Omega)} &\leq & C e^{-\hat \lambda_2 T}, \qquad \forall T\geq 1.
 \eean
 Applying inequalities \eqref{ineqEi},  \eqref{estimFeq} and \eqref{FF1} to the relation \eqref{MM1}, yields
 \bea
 |\frac{1}{\hat \lambda_1}-\frac{1}{\widehat{\tilde \lambda_1}}|\left| \int_\Gamma  \a_0 \partial_\nu u(., T) ds(x)\right|\leq  
C \left[\|u-\tilde u \|_{L^2(\Omega)} + e^{-\hat \lambda_2 T}\left(|\frac{1}{\hat \lambda_1}-\frac{1}{\widehat{\tilde \lambda_1}}| + 
\|\a-\ta\|_{L^2(\Omega)} \right)  \right].
 \eea
 We deduce from inequalities \eqref{inn}, \eqref{NN2}, and the fact that $\a_0\geq 1$, the following estimate
 \bea
 e^{-\hat \lambda_1 T} |\frac{1}{\hat \lambda_1}-\frac{1}{\widehat{\tilde \lambda_1}}|\leq  
C \left[\|u-\tilde u \|_{L^2(\Omega)} + e^{-\hat \lambda_2 T}\left(|\frac{1}{\hat \lambda_1}-\frac{1}{\widehat{\tilde \lambda_1}}| + 
\|\a-\ta\|_{L^2(\Omega)} \right)  \right].
 \eea 
 Inequality \eqref{gap2} gives
\bea
(1- Ce^{-\theta T}) |\frac{1}{\hat \lambda_1}-\frac{1}{\widehat{\tilde \lambda_1}}|\leq  
C \left[ e^{\hat \lambda_1 T} \|u-\tilde u \|_{L^2(\Omega)} + e^{-(\hat \lambda_2-\hat \lambda_1) T}
\|\a-\ta\|_{L^2(\Omega)}  \right],
 \eea 
with $\theta = \frac{\delta}{a_+ \lambda_1^{\Omega}}.$ Finally,  taking  $T \geq \max(T_{21}, 1), $ where
$T_{21}\in \mathbb R$ verifies $e^{-\theta T_{21}} = \frac{1}{2C}$, provides  the first inequality \eqref{XX1}. \\

Now, we shall focus on the second inequality.  Let 
\bea
\zeta =  \frac{1}{\a}\left(\frac{\a}{\hat \lambda_1}-\frac{\ta}{\widehat{\tilde \lambda_1}}\right).
\eea

 Multiplying again  the equation \eqref{aa} by $\zeta(x) u(x, T) $,  and integrating over $\Omega$, we obtain
\bea
- \frac{1}{2} \int_{\Omega} \a u(\cdot, T) \nabla  u(\cdot, T) \cdot \nabla \zeta^2 dx  -
\int_{\Omega} \zeta^2 \a  |\nabla  u(\cdot, T)|^2 dx+
\int_\Gamma  \a \zeta^2 u(\cdot, T)  \partial_\nu  u(\cdot, T) ds(x)\\
= \int_{\Omega}  
 \div( \frac{\ta}{\widehat{\tilde \lambda}_1}\nabla ( u(\cdot, T) -\tilde u(\cdot, T)))   \zeta u(\cdot, T) dx  +
   \int_\Omega(\tilde u(\cdot, T) - u(\cdot, T)) \zeta u(\cdot, T) dx \\+ \int_{\Omega} \left( \frac{1}{ \hat  \lambda_1} \F(\a; \cdot, T)- \frac{1}{\widehat{\tilde \lambda}_1}\F(\ta; \cdot, T)\right) \zeta u(\cdot, T) dx.
  \nonumber
 \eea

Integrating by parts the first term on the left hand side, leads to 
\bea
- \frac{1}{2} \int_{\Omega} \zeta^2  u(\cdot, T) \partial_t u(\cdot, T) dx  +
\frac{1}{2} \int_{\Omega} \zeta^2 \a  |\nabla  u(\cdot, T)|^2 dx-
\int_\Gamma  \a \zeta^2 u(\cdot, T)  \partial_\nu  u(\cdot, T) ds(x)\\\leq 
C\left[  \|u-\tilde u \|_{H^2(\Omega)}+  |\frac{1}{\hat \lambda_1}-\frac{1}{\widehat{\tilde \lambda_1}}| \|\F(\a; \cdot, T) \|_{L^2(\Omega)}
+ \frac{1}{\widehat{\tilde \lambda_1}} \|\F(\a; \cdot, T) - \F(\ta; \cdot , T)\|_{L^2(\Omega)}\right]\|u(\cdot, T) \|_{L^2(\Omega)}\|\zeta \|_{L^2(\Omega)}.
  \nonumber
 \eea
Since the third term on the right side is positive for $T\geq T_1$, we deduce from inequalities \eqref{FF1},  \eqref{ineqEi},  \eqref{estimFeq},   
the following estimate
\bean \label{WW2}
 -\frac{1}{2} \int_{\Omega} \zeta^2 u(\cdot, T) \partial_t u(\cdot, T) dx+\frac{1}{2} \int_{\Omega} \zeta^2 \a  |\nabla  u(\cdot, T)|^2 dx \\
C \left[\|u-\tilde u \|_{H^2(\Omega)} + e^{-\hat \lambda_2 T}\left(|\frac{1}{\hat \lambda_1}-\frac{1}{\widehat{\tilde \lambda_1}}| + 
\|\a-\ta\|_{L^2(\Omega)} \right)  \right] \|u(\cdot, T) \|_{L^2(\Omega)} \|\zeta \|_{L^2(\Omega)}, \quad \forall T\geq T_1. \nonumber
\eean
 
 On the other hand, we have
 \bea
 \|u \|_{L^2(\Omega)}^2 = \sum_{k=1}^\infty \hat \lambda_k^2e^{-2\hat \lambda_k T} \|P_k u_0\|_{L^2(\Omega)}^2
 \leq  e^{-2\hat \lambda_1 T} \sum_{k=1}^\infty \hat \lambda_k^2e^{-2(\hat \lambda_k-1)} \|u_0\|_{L^2(\Omega)}^2.
 \eea
 We deduce from inequalities \eqref{ineqEi}
 \bean \label{ZZ1}
 \|u \|_{L^2(\Omega)} &\leq & Ce^{-\hat \lambda_1 T}, \qquad \forall T\geq 1.
 \eean

Since $\frac{1}{\a} \geq \frac{1}{\a_+}$, we have 
\bean \label{ZZ2}
\| \zeta\|_{L^2(\Omega)}\leq  C\left\|  \frac{\a}{\hat \lambda_1}-\frac{\ta}{\widehat{\tilde \lambda_1}}\right\|_{L^2(\Omega)}.
\eean

Combining inequalities \eqref{NN0}, \eqref{NN1}, \eqref{NN2}, 
 \eqref{WW2},  \eqref{ZZ1},  and  \eqref{ZZ2},  we finally  obtain the second inequality \eqref{XX2} for $T\geq T_2= \max(T_1, T_{21})$.\\

\end{proof}

\begin{proposition} \label{proPP} Let $\a\in \A $, and 
  $\phi_1$ be the first eigenfunction of $\L$. Then for $\varepsilon \in ]0, \varepsilon_0[, $ there exists a constant 
  $C>0$ depending only  on $\varepsilon, \A, n, $ and  $\Omega$ such that 
 \bean \label{EigII}
 \phi_1^2(x) + \mathbbm{1}_{\Omega_{\varepsilon}}(x)|\nabla \phi_1(x)|^2 \geq  C,   \qquad  \forall x\in \Omega,
   \eean
   
 \end{proposition}
\begin{proof}
Since $\varepsilon \in ]0, \varepsilon_0[$,  inequality \eqref{invar} in Proposition \ref{1eigen} implies 
\bean \label{invar2}
|\nabla \phi_1(x)| &\geq& C_0, \qquad \forall x\in \Omega_\varepsilon.
\eean
Combining  \eqref{invar2} with inequality \eqref{inn} provides the wanted estimate.
\end{proof}
Back to the proof of  Theorem \ref{main}. Fixing $\varepsilon = \frac{\varepsilon_0}{2}$,
we deduce from Propositions \ref{mainProp} and \ref{proPP} the following estimate
\bean \label{XXX1}
\left\|  \frac{\a}{\hat \lambda_1}-\frac{\ta}{\widehat{\tilde \lambda_1}}\right\|_{L^2(\Omega)}\\ 
 \leq C\left[e^{\hat \lambda_1 T }\|u-\tilde u \|_{H^2(\Omega)}
 +e^{-(\hat \lambda_2 - \hat \lambda_1)T}\left(\|\a-\ta\|_{L^2(\Omega)}+   
 |\frac{1}{\hat \lambda_1}-\frac{1}{\widehat{\tilde \lambda_1}}| \right)
 \right],\nonumber
\eean
for all $T\geq T_2$. \\

We further assume that $T\geq T_2$. By a simple calculation, we get 
\bea
\left( \frac{1}{\widehat{\tilde  \lambda_1}} \int_{\Omega}| \a-\ta|^2 dx\right)^{\frac{1}{2} } 
 &\leq&\left\|  \frac{\a}{\hat \lambda_1}-\frac{\ta}{\widehat{\tilde \lambda_1}}\right\|_{L^2(\Omega)}
 +\ \left|\frac{1}{\hat \lambda_1}-\frac{1}{\widehat{\tilde \lambda_1}} \right| \|\ta\|_{L^2(\Omega)}.
\eea

We deduce from inequalities  \eqref{ineqEi}, the following estimate 
\bean \label{STG}
 \|\a-\ta\|_{L^2(\Omega)}
 &\leq& C \left[ \left\|  \frac{\a}{\hat \lambda_1}-\frac{\ta}{\widehat{\tilde \lambda_1}}\right\|_{L^2(\Omega)}
 +  \left|\frac{1}{\hat \lambda_1}-\frac{1}{\widehat{\tilde \lambda_1}} \right| \right].
\eean

Combining estimates  \eqref{STG},  \eqref{XXX1}, and  \eqref{XX1}, yields

\bean \label{mainR2}
 \|\a-\ta\|_{L^2(\Omega)}
 &\leq& C\left[e^{\hat \lambda_1 T } \|u-\tilde u \|_{H^2(\Omega)} 
 +e^{-(\hat \lambda_2 - \hat \lambda_1) T}\|\a-\ta\|_{L^2(\Omega)}
 \right].
\eean

 The gap condition \eqref{gap2} gives
\bea
(1- Ce^{-\theta T})  \|\a-\ta\|_{L^2(\Omega)}  &\leq &
C e^{\hat \lambda_1 T} \|u-\tilde u \|_{L^2(\Omega)} 
 \eea 
with $\theta = \frac{\delta}{a_+ \lambda_1^{\Omega}}.$ Finally,  taking  $T \geq \max(T_{2}, T_3), $ where
$T_{3}\in \mathbb R$ verifies $e^{-\theta T_{3}} = \frac{1}{2C}$, provides  the main estimate \eqref{stability} of Theorem~\ref{main}.

 \appendix
\section{}
\label{App} 
We  recall some known properties of the eigenelements of the unbounded  operator $L_a$.
 \begin{lemma}
 The eigenvalue $\lambda_1$ is simple, and has a strictly positive  eigenfunction $\phi_1 \in C^1(\overline \Omega)$.
  \end{lemma}
  \begin{proof}  The proof can be found in many references \cite{GT15, He06}. Since it is too 
  short and for the sake of completeness we give it here.\\
  
   We can recover the second result by using the Min-max principle.  It is well known that  the smallest
    eigenvalue $\lambda_1$ is the minimizer  of  the 
  Rayleigh quotient \cite{He06}  
 \[
  \lambda_1 =  \min_{\phi \in H_0^1(\Omega) \setminus\{0\}}  \frac{\int_\Omega a |\nabla \phi|^2 dx}{ \int_{\Omega} |\phi|^2 dx}.
 \]
 Since $\phi_1  \in H_0^1(\Omega) $ we also
have $|\phi_1| \in H^1(\Omega)$ and $\nabla |\phi_1| = \textrm{sign}(\phi_1) \nabla \phi_1$, we see
 that $|\phi_1|$ and $\phi_1$ has the similar
Rayleigh quotient. Therefore,  $|\phi_1|$ is also a minimizer of the Rayleigh quotient and, therefore, an 
eigenfunction associated to $\lambda_1$.  By Harnack inequality  for elliptic operators,  $|\phi_1|$ does not 
vanish  in $\Omega$. Since two functions having contant signs 
  can not be orthogonal in  $L^2(\Omega)$,  $\lambda_1$ is simple. We also deduce from elliptic 
  regularity that $\phi_1$ is $C^1(\overline \Omega)$.
 \end{proof}
 
 The proof of the following results can be found in Proposition 2.1 and Lemma 2.1 in \cite{ACT19} or Lemma 4.6.1 in \cite{Da90}.
 \begin{lemma} \label{lemA2} Let $\a\in \A $, and 
  $\phi_1$ be the first eigenfunction of $\L$. Then there exists a constant $C>0$ that  depends only  on $\A$ and 
 $\Omega$ such that  
 \bean \label{inn}
 \phi_1(x) \geq  C d_\Omega(x) \qquad  \forall x\in \Omega; \qquad
  -\partial_\nu \phi_1(x) > C \qquad  \forall x\in \partial \Omega.
  \eean
 \end{lemma}
  The proof of the  following lemma based on the  Min-max principle is forward. 
 \begin{lemma} \label{lemA3} Let $\a\in \A $, and  let  $\lambda_k, \,  k\in \mathbb N^*$,  be the increasing eigenvalues of  $\L$. 
   Then
  \bean \label{ineqEi}
\lambda_k^\Omega \leq \lambda_k \leq  a_+ \lambda_k^\Omega, \qquad \forall k \in \mathbb N^*,
  \eean
  where $\lambda_k^\Omega,  k\in \mathbb N^*$,  are the increasing Dirichlet  eigenvalues of  the  Laplacian $-\Delta$ in $\Omega$. 
 \end{lemma}
 The proof of the following lemma is based on the analysis of the rate of decay of the heat kernel, and  can be found in (Corollary  4.6.3 of \cite{Da90}).
 
 \begin{lemma} \label{lemA4}
 Let $\a\in \A $, and  let  $\lambda_k, \,  k\in \mathbb N^*$,  be the increasing eigenvalues of  $\L$, and  $\phi_k, \,  k\in \mathbb N^*,$  be corresponding orthonormal sequence of eigenfunctions. Then there exists a constant $C>0$ depending on $\A, n$ and 
 $\Omega$ such that  
 \bean\label{ineqEiG}
 |\phi_k (x)| \leq C \lambda_k^{\frac{1}{2}+\frac{n}{4}} \phi_1(x), \qquad \forall x \in \Omega, \;   k\in \mathbb N^*.
 \eean
 
 \end{lemma}


\end{document}